\documentclass[10pt,a4paper,twoside]{amsart}

\usepackage{graphicx}
\usepackage{a4wide}

\newtheorem{definition}{Definition}[section]
\newtheorem{theorem}{Theorem}[section]
\newtheorem{lemma}{Lemma}[section]

\newtheorem{remark}{Remark}[section]
\newtheorem*{maintheorem*}{Main Theorem}
\allowdisplaybreaks
\numberwithin{equation}{section}

\newcommand{\norm}[1]{\left\| #1 \right\|}

\newcommand{\eps}{\varepsilon}

\newcommand{\eb}{{\eps,\beta}}
\newcommand{\ueb}{u_\eb}

\newcommand{\ue}{u_\eps}

\newcommand{\Peb}{P_{\eps,\beta}}
\newcommand{\Feb}{F_{\eps,\beta}}
\newcommand{\Pe}{P_\eps}

\newcommand{\pt}{\partial_t}

\newcommand{\px}{\partial_x }
\newcommand{\pxx}{\partial_{xx}^2}
\newcommand{\pxxx}{\partial_{xxx}^3}
\newcommand{\pxxxx}{\partial_{xxxx}^4}

\newcommand{\ptx}{\partial_{tx}^2}

\renewcommand{\i}{\ifmmode\mathit{\mathchar"7010 }\else\char"10 \fi}
\renewcommand{\j}{\ifmmode\mathit{\mathchar"7011 }\else\char"11 \fi}
\newcommand{\R}{\mathbb{R}}
\newcommand{\N}{\mathbb{N}}

{%

\begin{enumerate}}%
{\end{enumerate}}

{%

\begin{enumerate}}%
{\end{enumerate}}

\begin{document}\large

\title[The Ostrovsky Equation]{Convergence of the Ostrovsky Equation \\ to the  Ostrovsky--Hunter One}

\author[G. M. Coclite and L. di Ruvo]{Giuseppe Maria Coclite and Lorenzo di Ruvo}
\address[Giuseppe Maria Coclite and Lorenzo di Ruvo]
{\newline Department of Mathematics,   University of Bari, via E. Orabona 4, 70125 Bari,   Italy} 
\email[]{giuseppemaria.coclite@uniba.it, lorenzo.diruvo@uniba.it}
\urladdr{http://www.dm.uniba.it/Members/coclitegm/}

\date{\today}

\keywords{Singular limit, compensated compactness, Ostrovsky-Hunter equation, Entropy condition.}

\subjclass[2000]{35G25, 35L65, 35L05}


\thanks{The authors would like to thank Prof. Kenneth Hvistendahl Karlsen for many useful discussions.}

\begin{abstract}
We consider the Ostrovsky equation, which contains nonlinear dispersive effects. We prove
that as the diffusion parameter tend to zero, the solutions of the dispersive equation converge to discontinuous weak solutions of the Ostrovsky-Hunter equation.
The proof relies on deriving suitable a priori estimates together with an application of the compensated compactness method in the $L^p$ setting.
\end{abstract}

\maketitle


\section{Introduction}
\label{sec:intro}

The nonlinear evolution equation
\begin{equation}
\label{eq:OHbeta}
\px(\pt u+u\px u-\beta \pxxx u)=\gamma u,
\end{equation}
with $\beta$ and $\gamma \in \R$ was derived by Ostrovsky \cite{O} to model small-amplitude long
waves in a rotating fluid of a finite depth. This equation generalizes the
Korteweg-deVries equation (that corresponds to $\gamma=0$) by the additional term induced by
the Coriolis force.  Mathematical properties of the Ostrovsky equation \eqref{eq:OHbeta} were studied
recently in many details, including the local and global well-posedness in energy space
\cite{GL, LM, LV, T}, stability of solitary waves \cite{LL, L, LV:JDE}, and convergence of solutions in the
limit of the Korteweg-deVries equation \cite{LL:07, LV:JDE}.
We rewrite \eqref{eq:OHbeta} in the following way
\begin{equation}
\label{eq:OHw-u}
\begin{cases}
\pt u+ u\px u- \beta \pxxx u=\gamma \int^{x}_{0} u(t,y) dy,&\qquad t>0, \ x\in\R,\\
u(0,x)=u_0(x), &\qquad x\in\R,
\end{cases}
\end{equation}
or equivalently,
\begin{equation}
\label{eq:OHw}
\begin{cases}
\pt u+ u\px u - \beta\pxxx u=\gamma P,&\qquad t>0, \ x\in\R ,\\
\px P=u,&\qquad t>0, \ x\in\R,\\
 P(t,0)=0,& \qquad t>0,\\
u(0,x)=u_0(x), &\qquad x\in\R.
\end{cases}
\end{equation}

We are interested in the no high frequency limit,  we send $\beta\to 0$ in \eqref{eq:OHbeta}. In this way we pass from \eqref{eq:OHbeta} to the equation
\begin{equation}
\label{eq:OH}
\px(\pt u+u\px u)=\gamma u,\qquad t>0, \quad x\in\R.
\end{equation}
The equation
\eqref{eq:OH} is known under different names such as the reduced Ostrovsky equation \cite{P, S}, the
Ostrovsky-Hunter equation \cite{B}, the short-wave equation \cite{H}, and the Vakhnenko equation
\cite{MPV, PV}. Integrating \eqref{eq:OH} with respect to $x$ we gain the integro-differential formulation of  \eqref{eq:OH} (see \cite{LPS})
\begin{equation*}
\pt u+\px f(u)=\gamma \int^x u(t,y) dy,
\end{equation*}
that is equivalent to
\begin{equation*}
\pt u+\px f(u)=\gamma P,\qquad \px P=u.
\end{equation*}

On the initial datum, we assume that
\begin{equation}
\label{eq:assinit}
u_0\in L^2(\R)\cap L^{4}(\R),\quad\int_{\R}u_{0}(x)dx=0,
\end{equation}
and on the function
\begin{equation}
\label{eq:def-di-P0}
P_{0}(x)=\int_{-\infty}^{x} u_{0}(y)dy, \quad x\in\R,
\end{equation}
we assume that
\begin{equation}
\label{eq:L-2P0}
\begin{split}
\norm{P_0}^2_{L^2(\R)}&=\int_{\R}\left(\int_{-\infty}^{x}u_{0}(y)dy\right)^2dx <\infty,\\
\int_{\R}P_0(x)dx&= \int_{\R}\left(\int_{-\infty}^{x}u_{0}(y)dy\right)dx=0.
\end{split}
\end{equation}

\begin{definition}
\label{def:sol}
We say that $u\in  L^{\infty}((0,T)\times\R),\,T>0,$ is an entropy solution of the initial
value problem \eqref{eq:OHw} if
\begin{itemize}
\item[$i$)] $u$ is a distributional solution of \eqref{eq:OHw};
\item[$ii$)] for every convex function $\eta\in  C^2(\R)$ the
entropy inequality
\begin{equation}
\label{eq:OHentropy}
\pt \eta(u)+ \px q(u)-\gamma\eta'(u) P\le 0, \qquad     q(u)=\int^u f'(\xi) \eta'(\xi)\, d\xi,
\end{equation}
holds in the sense of distributions in $(0,\infty)\times\R$.
\end{itemize}
\end{definition}
In \cite{CdK, CdRK, dR}, it is proved that \eqref{eq:OHw} has an entropy solution in the sense of Definition \ref{def:sol}. Moreover, it is unique and stable.

We study the dispersion-diffusion for \eqref{eq:OH}. Therefore, we consider the following third order approximation
\begin{equation}
\label{eq:OHepsw}
\begin{cases}
\pt \ueb+\ueb\px \ueb -\beta\pxxx\ueb=\gamma\Peb+ \eps\pxx\ueb,&\quad t>0,\ x\in\R,\\
\px\Peb=\ueb,&\quad t>0,\ x\in\R,\\
\Peb(t,0)=0,&\quad t>0,\\
\ueb(0,x)=u_{\eps,\beta,0}(x),&\quad x\in\R,
\end{cases}
\end{equation}
or equivalently,
\begin{equation}
\label{eq:OHepswint}
\begin{cases}
\pt \ueb+\ueb\px \ueb - \beta\pxxx\ueb\\
\qquad=\gamma\int_{0}^x \ueb (t,y)dy+ \eps\pxx\ueb,&\quad t>0,\ x\in\R ,\\
\ueb(0,x)=u_{\eps, \beta, 0}(x),&\quad x\in\R,
\end{cases}
\end{equation}
where $u_{\eps,\beta,0}$ is a $C^\infty$ approximation of $u_{0}$ such that
\begin{equation}
\begin{split}
\label{eq:u0eps}
&\norm{u_{\eps,\beta, 0}}^2_{L^2(\R)}+ \norm{u_{\eps,\beta, 0}}^4_{L^4(\R)}+(\beta+ \eps^2) \norm{\px u_{\eps,\beta,0}}^2_{L^2(\R)}\\
&\qquad\qquad\quad\quad +\beta^2\norm{\pxx u_{\eps,\beta, 0}}^2_{L^2(\R)}\le C_{0}, \quad \eps,\beta >0,\\
& \int_{\R} u_{\eps,\beta,0}(x) dx =0,\quad \beta\int_{\R}\ueb(\px\ueb)^2 dx \leq C_{0}, \quad \eps,\beta >0,\\
& \int_{\R}\left(\int_{-\infty}^{x}u_{\eps, \beta, 0}(y)dy\right) dx=\int_{\R}P_{0,\eps, \beta}(x)dx=0, \quad \eps,\beta>0,\\
&\int_{\R}\left(\int_{-\infty}^{x}u_{\eps, \beta, 0}(y)dy\right)^2 dx\le \norm{P_0}^2_{L^2(\R)} \quad \eps,\beta >0,
\end{split}
\end{equation}
where $C_0$ is a constant independent on $\eps$ and $\beta$.\\

The main result of this paper is the following theorem.

\begin{theorem}
\label{th:main}
Assume that \eqref{eq:assinit}, \eqref{eq:def-di-P0}, \eqref{eq:L-2P0}, and \eqref{eq:u0eps} hold.
If
\begin{equation}
\label{eq:beta-eps}
\beta=\mathbf{\mathcal{O}}(\eps^2),
\end{equation}
then, there exist two sequences $\{\eps_{n}\}_{n\in\N}$, $\{\beta_{n}\}_{n\in\N}$, with $\eps_n, \beta_n \to 0$, and two limit functions
\begin{align*}
&u\in L^{\infty}(0,T; L^4(\R)\cap L^2(\R)), \quad T>0,\\
&P\in L^{\infty}((0,T)\times\R)\cap L^{2}((0,T)\times\R), \quad T>0,
\end{align*}
such that
\begin{itemize}
\item[$i)$] $u_{\eps_n, \beta_n}\to u$ strongly in $L^{p}_{loc}((0,T)\times\R)$, for each $1\le p <4$, $T>0$,
\item[$ii)$] $P_{\eps_n, \beta_n}\to P$ strongly in $L^{\infty}((0,T)\times\R)\cap L^{2}((0,T)\times\R)$, $T>0$,
\end{itemize}
where $u$ is a distributional solution of \eqref{eq:OHw}. Moreover, if
\begin{equation}
\label{eq:beta-eps-1}
\beta=o(\eps^2),
\end{equation}
then,
\begin{itemize}
\item[$iii)$] $u$ is  the unique entropy solution of \eqref{eq:OHw},
\item[$iv)$] $P_{\eps_n, \beta_n}\to P$ strongly in $L^{p}_{loc}(0,T;W^{1,\infty}(\R)\cap H^{1}(\R))$, for each $1\le p <4$, $T>0$.
\end{itemize}
\end{theorem}

The paper is organized in two sections. In Section \ref{sec:vv}, we prove some priori estimates, while in Section \ref{sec:theor} we prove Theorem \ref{th:main}.

\section{A priori Estimates}
\label{sec:vv}
Let us prove some a priori estimates on $\ueb$, denoting with $C_0$ the constants which depend only on the initial data, and with $C(T)$, the constants which depend also on $T$.
\begin{lemma}

\label{lm:cns}
Let us suppose that
\begin{equation}
\label{eq:Pu}
\Peb(t,-\infty)=0, \quad t\ge 0,\quad (\textrm{or} \quad \Peb(t,\infty)=0),
\end{equation}
where $\Peb(t,x)$ is defined in \eqref{eq:OHepsw}. Then, the following statements are equivalent
\begin{align}
\label{eq:con-u}
\int_{\R} \ueb(t,x) dx &=0, \quad t\ge 0,\\
\label{eq:120}
\frac{d}{dt}\int_{\R} \ueb^2dx + 2\eps\int_{\R}
(\px\ueb)^2 dx&=0,\quad t>0.
\end{align}
\end{lemma}

\begin{proof}
Let $t>0$. We begin by proving that \eqref{eq:con-u} implies
\eqref{eq:120}. \\
Multiplying \eqref{eq:OHepswint} by $\ueb$, an integration over $\R$ gives
\begin{align*}
\frac{1}{2}\frac{d}{dt}\int_{\R} \ueb^2dx=&\int_{\R} \ueb\pt\ueb dx\\
=&\eps\int_{\R}\ueb\pxx\ueb dx-\int_{\R} \ueb \ueb\px\ueb dx\\
&+\gamma\int_{\R}\ueb\Big(\int_{0}^{x}\ueb dy\Big)dx+\beta\int_{\R} \pxxx\ueb\ueb dx\\
= & - \eps\int_{\R} (\px\ueb)^2 dx + \gamma\int_{\R}\ueb\Big(\int_{0}^{x}\ueb dy\Big)dx  -\beta\int_{\R}\pxx\ueb\px\ueb dx\\
=&- \eps\int_{\R} (\px\ueb)^2 dx + \gamma\int_{\R}\ueb\Big(\int_{0}^{x}\ueb dy\Big)dx.
\end{align*}
For \eqref{eq:OHepsw} and \eqref{eq:Pu},
\begin{align*}
\int_{\R}\ueb\Big(\int_{0}^{x}\ueb dy\Big)dx=&\int_{\R}\Peb(t,x)\px\Peb(t,x)dx\\
=&\frac{1}{2}(\Peb^2(t,\infty)-\Peb ^2(t,-\infty))
=\frac{1}{2}\Peb^2(t,\infty).
\end{align*}
Then,
\begin{equation}
\label{eq:121}
\frac{d}{dt}\int_{\R} \ueb^2dx + 2\eps\int_{\R} (\px\ueb)^2 dx=\gamma\Peb ^2(t,\infty).
\end{equation}
Again by \eqref{eq:OHepsw},
\begin{equation*}
\Peb(t,\infty)=\int_{0}^{\infty}\ueb(t,x)dx.
\end{equation*}
Therefore, \eqref{eq:Pu} and \eqref{eq:con-u} give
\begin{equation}
\label{eq:122}
\Peb(t,\infty)=\Peb(t,-\infty)=0.
\end{equation}
\eqref{eq:120} follows from \eqref{eq:121}, and \eqref{eq:122}.

Let us show that \eqref{eq:120} implies \eqref{eq:con-u}. 
We assume by contradiction that \eqref{eq:con-u} does not hold, namely:
\begin{equation*}
\int_{\R} \ueb(t,x) dx \neq 0.
\end{equation*}
For \eqref{eq:OHepsw}, and \eqref{eq:Pu},
\begin{align*}
\Peb(t,\infty)=\int_{0}^{\infty} \ueb(t,x) dx\neq \int_{0}^{-\infty}
\ueb(t,x) dx = \Peb(t, -\infty)= 0.
\end{align*}
Then,
\begin{equation*}
\Peb^2(t,\infty)=\Big(\int_{0}^{\infty} \ueb(t,x) dx\Big)^2\neq 0.
\end{equation*}
Therefore, \eqref{eq:121} gives
\begin{equation*}
\frac{d}{dt}\int_{\R} \ueb^2dx + 2\eps\int_{\R}
(\px\ueb)^2 dx\neq 0,
\end{equation*}
which is in contradiction with \eqref{eq:120}.
\end{proof}

\begin{lemma}
\label{lm:l2-u}
For each $t \ge 0$, \eqref{eq:con-u} holds true, and
\begin{equation}
\begin{split}
\label{eq:P0}
\Peb(t,\infty)&=\int_{0}^{-\infty}\ueb(t,x)dx=0,\\
\Peb(t,-\infty)&=\int_{0}^{\infty}\ueb(t,x) =0.
\end{split}
\end{equation}
In particular, we have that
\begin{equation}
\label{eq:l2-u}
\norm{\ueb(t,\cdot)}_{L^2(\R)}^2+2\eps \int_0^t \norm{\px \ueb(s,\cdot)}^2_{L^2(\R)}ds\le C_0.
\end{equation}
\end{lemma}
\begin{remark}\label{Remark-1}
In light of \eqref{eq:P0}, we have
\begin{equation}
\label{eq:new-def-P}
\Peb(t,x)=\int_{0}^{x} \ueb(t,y)dy=\int_{-\infty}^{x}\ueb(t,y)dy.
\end{equation}
\end{remark}
\begin{proof}[Proof of Lemma \ref{lm:l2-u}]
Differentiating \eqref{eq:OHepsw} with respect to $x$, we have
\begin{equation*}
\px(\pt\ueb + \ueb \px\ueb-\beta\pxxx\ueb - \eps \pxx \ueb)= \gamma\ueb.
\end{equation*}
Since $\ueb$ is a smooth solution of \eqref{eq:OHepsw}, an integration over $\R$ gives \eqref{eq:con-u}.
Again for the regularity of $\ueb$, from \eqref{eq:OHepsw}, we get
\begin{align*}
&\lim_{x\to -\infty}(\pt\ueb+ \ueb\px \ueb-\beta\pxxx\ueb -\eps\pxx\ueb)\\
&\qquad\qquad\qquad=\gamma\int_{0}^{-\infty}\ueb(t,x)dx=\gamma\Peb(t,-\infty)=0,\\
&\lim_{x\to \infty}(\pt\ueb+\ueb \px\ueb-\beta\pxxx\ueb -\eps\pxx\ueb)\\
&\qquad\qquad\qquad =\gamma\int_{0}^{\infty}\ueb(t,x)dx=\gamma\Peb(t,\infty)=0,
\end{align*}
which gives \eqref{eq:P0}.

Lemma \ref{lm:cns} says that also \eqref{eq:120} holds true.
Therefore, integrating \eqref{eq:120} on $(0,t)$, for \eqref{eq:u0eps},
we have \eqref{eq:l2-u}.
\end{proof}
Let us observe that, from the second equation of \eqref{eq:OHepsw}, we get
\begin{equation}
\label{eq:pt}
\pt\Peb(t,x)=\frac{d}{dt}\int_{0}^{x}\ueb(t,y)dy=\int_{0}^{x}\pt\ueb(t,y) dy.
\end{equation}
Therefore, integrating  on $(0,x)$ the first equation of \eqref{eq:OHepsw}, we obtain
\begin{equation}
\begin{split}
\label{eq:P1}
\pt\Peb &+\frac{1}{2}\ueb^2(t,x) -\eps\px\ueb(t,x)-\beta\pxx\ueb(t,x)\\
&- \frac{1}{2}\ueb^2(t,0))+\eps\px\ue(t,0)+\beta \pxx\ueb(t,0)=\gamma\int_{0}^{x}\Peb(t,y)dy.
\end{split}
\end{equation}
\begin{lemma}
\label{lm:p8}
For each $t\ge 0$, we have that
\begin{align}
\label{eq:intp-infty}
\int_{0}^{-\infty}\Peb(t,x)dx&=a_{\eps, \beta}(t), \\
\label{eq:int+infty}
\int_{0}^{\infty}\Peb(t,x)dx&=a_{\eps, \beta}(t),
\end{align}
where
\begin{equation}
a_{\eps,\,\beta}(t)=-\frac{1}{2\gamma}\ueb^2(t,0)) + \frac{\eps}{\gamma}\px\ueb(t,0)+\frac{\beta}{\gamma}\pxx\ueb(t,0).
\end{equation}
Moreover,
\begin{equation}
\label{eq:Pmedianulla}
\int_{\R}\Peb(t,x)dx=0, \quad t\geq 0.
\end{equation}
\end{lemma}
\begin{proof}
We begin by observing that, for \eqref{eq:P0}, we have that
\begin{equation}
\label{eq:Pe10}
\lim_{x\to-\infty} \pt\Peb(t,x)=\int_{0}^{-\infty} \pt\ueb(t,x)dx=\frac{d}{dt}\int_{0}^{-\infty}\ueb(t,x)dx=0.
\end{equation}
Since
\begin{equation}
\label{eq:u-in-h-2}
\ueb(t,\cdot)\in H^2(\R), \quad t>0,
\end{equation}
we get
\begin{equation}
\label{eq:500}
\lim_{x\to-\infty}\left( \frac{1}{2}\ueb^2(t,x)-\eps\px\ueb(t,x)-\beta\pxx\ueb(t,x) \right)=0.
\end{equation}
Therefore, \eqref{eq:intp-infty} follows from \eqref{eq:P1}, \eqref{eq:Pe10} and \eqref{eq:500}.

Again by \eqref{eq:P0},
\begin{equation}
\label{eq:Pe11}
\lim_{x\to\infty} \pt\Peb(t,x)=\int_{0}^{\infty} \pt\ueb(t,x)dx=\frac{d}{dt}\int_{0}^{\infty}\ueb(t,x)dx=0.
\end{equation}
Moreover, again by \eqref{eq:u-in-h-2},
\begin{equation}
\label{eq:510}
\lim_{x\to\infty}\left( \frac{1}{2}\ueb^2(t,x)-\eps\px\ueb(t,x)- \beta\pxx\ueb(t,x)\right)=0.
\end{equation}
\eqref{eq:P1}, \eqref{eq:Pe11} and \eqref{eq:510} give \eqref{eq:int+infty}.

It follows from \eqref{eq:intp-infty} that
\begin{equation*}
\int_{-\infty}^{0}\Peb(t,x)dx = -a_{\eps, \beta}(t).
\end{equation*}
Therefore, for \eqref{eq:int+infty},
\begin{align*}
\int_{-\infty}^{0}\Peb(t,x)dx+\int_{0}^{\infty}\Peb(t,x)dx=\int_{\R} \Peb(t,x)dx =-a_{\eps,\beta}(t)+a_{\eps, \beta}(t)=0,
\end{align*}
that is \eqref{eq:Pmedianulla}.
\end{proof}
We observe that, for \eqref{eq:con-u}, \eqref{eq:intp-infty} and \eqref{eq:int+infty}, $\Peb(t,x)$  is integrable at $\pm\infty$.
Therefore, for each $t\ge 0$, we can consider the following function
\begin{equation}
\label{eq:F1}
\Feb(t,x)=\int_{-\infty}^{x}\Peb(t,y) dy.
\end{equation}
Thus, integrating the first equation of \eqref{eq:OHepsw} on $(-\infty,x)$, thanks \eqref{eq:F1}, we obtain
\begin{equation}
\label{eq:intP1}
\int_{-\infty}^{x}\pt\ueb dy+\frac{1}{2}\ueb^2 - \beta\pxx\ueb-\eps\px\ueb=\gamma\Feb.
\end{equation}
From the second equation of \eqref{eq:OHepsw} and Remark \ref{Remark-1}, we have
\begin{equation}
\label{eq:equaP}
\pt\Peb+ \frac{1}{2}\ueb^2-\beta\pxx\ueb -\eps\px\ueb=\gamma\Feb.
\end{equation}
\begin{lemma}
\label{lm:P-infty}
Let $T>0$. There exists a function $C(T)>0$, independent on $\eps$ and $\beta$, such that
\begin{align}
\label{eq:P-infty}
\norm{\Peb}_{L^{\infty}(I_T)}&\le C(T),\\
\label{eq:l2P}
\norm{\Peb(t,\cdot)}_{L^{2}(I_T)}&\le C(T),
\end{align}
where
\begin{equation}
\label{eq:defI}
I_{T}=(0,T)\times\R.
\end{equation}
\end{lemma}
\begin{proof}
Let $0\le t \le T$. We begin by proving that
\begin{equation}
\begin{split}
\label{eq:1200}
\norm{\Peb(t,\cdot)}^2_{L^2(\R)}=\norm{P_{\eps, \beta, 0}}^2_{L^2(\R)} &-2\eps\int_{0}^{t}\norm{\ueb(s,\cdot)}^2_{L^2(\R)}ds\\
&-\int_{0}^{t}\!\!\!\int_{\R}\Peb \ueb^2 dsdx.
\end{split}
\end{equation}
Multiplying  \eqref{eq:equaP} by $\Pe$, an integration on $\R$ gives
\begin{align*}
\frac{d}{dt}\int_{\R} \Peb^2dx=&\int_{\R} \Peb\pt\Peb dx\\
=&2\eps\int_{\R}\px\ueb\Peb dx  - \int_{\R}\Peb\ueb^2 dx \\
&+2\beta\int_{\R}\pxx\ueb\Peb dx +2\gamma\int_{\R}\Peb\Feb dx.
\end{align*}
For \eqref{eq:OHepsw},
\begin{equation}
\label{eq:2005}
\begin{split}
2\eps\int_{\R}\px\ueb\Peb dx&= -2\eps\int_{\R}\ueb\px\Peb dx=-2\eps\int_{\R}\ueb^2dx,\\
2\beta\int_{\R}\pxx\ueb\Peb dx &= -2\beta\int_{\R}\px\ueb\ueb dx=0,
\end{split}
\end{equation}
while, for \eqref{eq:F1},
\begin{equation*}
2\int_{\R}\Peb\Feb dx=2\int_{\R}\Feb\px\Feb dx= \Feb^2(t,\infty)-\Feb^2(t,-\infty)=\Feb^2(t,\infty).
\end{equation*}
Then,
\begin{equation}
\label{eq:300}
\frac{d}{dt}\int_{\R} \Peb^2dx=-2\eps\int_{\R}\ueb^2dx -\int_{\R} \Peb \ue^2 dx +\gamma\Feb ^2(t,\infty).
\end{equation}
It follows from \eqref{eq:Pmedianulla} that
\begin{equation}
\label{eq:301}
\lim_{x\to\infty}\Feb ^2(t,x)= \left(\int_{\R}\Peb(t,x)dx\right)^2=0.
\end{equation}
\eqref{eq:300}, \eqref{eq:301} and an integration on $(0,t)$ give \eqref{eq:1200}.

Let us show that
\begin{equation}
\label{eq:l2-P}
\norm{\Peb(t,\cdot)}^2_{L^2(\R)}\le\norm{P_{0}}^2_{L^2(\R)}+\int_{0}^{t}\!\!\!\int_{\R}\vert\Peb\vert \ueb^2 dsdx.
\end{equation}
Thanks to \eqref{eq:u0eps} and \eqref{eq:1200}, we have
\begin{align*}
\norm{\Peb(t,\cdot)}^2_{L^2(\R)}&\leq\norm{P_{0}}^2_{L^2(\R)}-\int_{0}^{t}\!\!\!\int_{\R}\Peb \ueb^2 dsdx\\
&\leq\norm{P_{0}}^2_{L^2(\R)}+\left\vert\int_{0}^{t}\!\!\!\int_{\R}\Peb \ueb^2 dsdx\right\vert\\
&\leq \norm{P_{0}}^2_{L^2(\R)} + \int_{0}^{t}\!\!\!\int_{\R}\vert\Peb\vert \ueb^2 dsdx,
\end{align*}
that is \eqref{eq:l2-P}.

We prove \eqref{eq:P-infty}. 
Due to the H\"older inequality, we get
\begin{equation*}
\Peb^2(t,x)\leq 2\int_{\R}\vert\Peb\px\Peb\vert dx\leq 2\norm{\Peb(t,\cdot)}_{L^2(\R)}\norm{\px\Peb(t,\cdot)}_{L^2(\R)},
\end{equation*}
that is
\begin{equation*}
\Peb^4(t,x)\leq 4\norm{\Peb(t,\cdot)}^2_{L^2(\R)}\norm{\px\Peb(t,\cdot)}^2_{L^2(\R)} .
\end{equation*}
For \eqref{eq:OHepsw}, \eqref{eq:l2-u} and \eqref{eq:l2-P},
\begin{align*}
\Peb^4(t,x)\leq& 4\norm{\Peb(t,\cdot)}^2_{L^2(\R)}C_{0}^2\\
\leq&4\left(\norm{P_{0}}^2_{L^2(\R)}+\int_{0}^{t}\!\!\!\int_{\R}\vert\Peb\vert\ueb^2dsdx\right)C_{0}^2\\
\leq &4 \norm{P_{0}}^2_{L^2(\R)}C_0^2 +4C_0\left(\int_{0}^{t}\!\!\!\int_{\R}\vert\Peb\vert\ueb^2dsdx\right)\\
\leq& 4 \norm{P_{0}}^2_{L^2(\R)}C_0^2 +4 C_0^4T\norm{\Peb}_{L^{\infty}(I_T)}.
\end{align*}
Therefore,
\begin{equation}
\label{eq:L-4}
\norm{\Peb}^4_{L^{\infty}(I_T)}-4C_0^4 T\norm{\Peb}_{L^{\infty}(I_T)} - 4 C_0^2 \norm{P_{0}}^2_{L^2(\R)}\leq 0.
\end{equation}
Let us consider the following function
\begin{equation}
\label{eq:g-func}
g(X)=X^4 - A(T)X -B,
\end{equation}
where
\begin{equation}
\label{eq:cost1}
A(T)=4C_0^4 T>0, \quad B=4 C_0^2 \norm{P_{0}}^2_{L^2(\R)}>0.
\end{equation}
We observe that
\begin{equation}
\label{eq:teozeri}
\lim_{X\to -\infty} g(X)=\infty, \quad g(0)=-B<0.
\end{equation}
Since $g'(X)=4 X^3 -A(T) $, we have that
\begin{align*}
g\quad \textrm{is increasing in}\quad [E(T), \infty),
\end{align*}
where $\displaystyle E(T)=\Big(\frac{A(T)}{4}\Big)^{\frac{1}{3}}>0$.\\
Thus,
\begin{equation}
\label{eq:600}
g(E(T))<g(0)<0.
\end{equation}
Moreover,
\begin{equation}
\label{eq:7000}
\lim_{X\to \infty} g(X)=\infty.
\end{equation}
Then, it follows from \eqref{eq:teozeri}, \eqref{eq:600} and \eqref{eq:7000} that the function $g$ has only two zeros $D(T)<0<C(T)$.
Therefore, the inequality
\begin{equation*}
X^4 - A(T)X -B\leq 0
\end{equation*}
is verified when
\begin{equation}
\label{eq:601}
D(T)\leq X\leq C(T).
\end{equation}
Taking $X=\norm{\Peb}_{L^{\infty}(I_T)}$, we have \eqref{eq:P-infty}.

Finally, \eqref{eq:l2P} follows from \eqref{eq:P-infty} and \eqref{eq:l2-P}.
\end{proof}

\begin{lemma}\label{lm::l-infty-u}
Fixed $T>0$. Then,
\begin{equation}
\label{eq:l-infty-u}
\norm{\ueb}_{L^{\infty}(I_T)}\le C(T)\beta^{-\frac{1}{3}},
\end{equation}
where $I_T$ is defined in \eqref{eq:defI}. Moreover, for every $0\le t\le T$,
\begin{equation}
\label{eq:ux2}
\beta\norm{\px\ueb(t,\cdot)}^2_{L^2(\R)}+ 2\beta\eps\int_{0}^{t} \norm{\pxx\ueb(s,\cdot)}^2_{L^2(\R)} ds \le C(T)\beta^{-\frac{1}{3}}.
\end{equation}
\end{lemma}
\begin{proof}
Let $0\le t\le T$. Multiplying  \eqref{eq:OHepsw} by $-2\beta\pxx\ueb + \ueb^2$, we have
\begin{equation}
\label{eq:Ohmp}
\begin{split}
&(-2\beta\pxx\ueb +\ueb^2)\pt\ueb + \frac{1}{2}(-2\beta\pxx\ueb +\ueb^2)\px\ueb^2\\
&\qquad-\beta(-2\beta\pxx\ueb +\ueb^2)\pxxx\ueb -\eps(-2\beta\pxx\ueb +\ueb^2)\pxx\ueb\\
&\quad =\gamma(-2\beta\pxx\ueb + \ueb^2)\Peb.
\end{split}
\end{equation}
Since
\begin{align*}
\int_{\R}(-2\beta\pxx\ueb +\ueb^2)\pt\ueb dx &= \frac{d}{dt}\left(\beta\norm{\px\ueb(t,\cdot)}^2_{L^2(\R)}+\frac{1}{3}\int_{\R}\ueb^3dx\right),\\
\int_{\R}(-2\beta\pxx\ueb +\ueb^2)\ueb\px\ueb dx &= -2\beta\int_{\R}\ueb\px\ueb\pxx\ueb dx,\\
-\beta\int_{\R}(-2\beta\pxx\ueb +\ueb^2)\pxxx\ueb dx &= 2\beta\int_{\R}\ueb\px\ueb\pxx\ueb dx,\\
-\eps\int_{\R}(-2\beta\pxx\ueb +\ueb^2)\pxx\ueb dx &= 2\beta\eps\norm{\pxx\ueb(t,\cdot)}^2_{L^2(\R)}+ 2\eps\int_{\R} \ueb(\px\ueb)^2dx,
\end{align*}
integrating \eqref{eq:Ohmp} over $\R$,
\begin{align*}
&\frac{d}{dt}\left(\beta\norm{\px\ueb(t,\cdot)}^2_{L^2(\R)}+\frac{1}{3}\int_{\R}\ueb^3dx\right) \\
& \qquad+ 2\eps\int_{\R} \ueb(\px\ueb)^2dx+ 2\beta\eps\norm{\pxx\ueb(t,\cdot)}^2_{L^2(\R)}\\
&\quad = -2\gamma\beta\int_{\R}\pxx\ueb\Peb dx  + \gamma\int_{\R}\ueb^2\Peb dx.
\end{align*}
Due to \eqref{eq:OHepsw},
\begin{equation*}
-2\gamma\beta\int_{\R}\pxx\ueb\Peb dx= 2\gamma\beta\int_{\R}\px\ueb\px\Peb dx= 2\gamma\beta\int_{\R}\ueb\px\ueb dx=0.
\end{equation*}
Therefore, thanks to  \eqref{eq:l2-u}, \eqref{eq:l2P} and the H\"older inequality, we get
\begin{align*}
&\frac{d}{dt}\left(\beta\norm{\px\ueb(t,\cdot)}^2_{L^2(\R)}+\frac{1}{3}\int_{\R}\ueb^3dx\right) + 2\beta\eps\norm{\pxx\ueb(t,\cdot)}^2_{L^2(\R)}\\
& \qquad = \gamma\int_{\R}\ueb^2\Peb dx - 2\eps\int_{\R} \ueb(\px\ueb)^2 dx\\
&\qquad \le \gamma\norm{\ueb}_{L^{\infty}(I_T)}\int_{\R}\vert\ueb\vert\vert\Peb\vert dx + 2\eps\norm{\ueb}_{L^{\infty}(I_T)}\int_{\R} (\px\ueb)^2dx\\
&\qquad \le \gamma C_0 C(T)\norm{\ueb}_{L^{\infty}(I_T)}+ 2\eps\norm{\ueb}_{L^{\infty}(I_T)}\int_{\R}(\px\ueb)^2dx.
\end{align*}
An integration over $(0,t)$, \eqref{eq:u0eps} and \eqref{eq:l2-u} give
\begin{align*}
&\beta\norm{\px\ueb(t,\cdot)}^2_{L^2(\R)}+ 2\beta\eps\int_{0}^{t}\norm{\pxx\ueb(s,\cdot)}^2_{L^2(\R)}ds \\
&\qquad \le C_{0} + \gamma C_0 C(T)\norm{\ueb}_{L^{\infty}(I_T)}t\\
&\qquad\quad +2\eps\norm{\ueb}_{L^{\infty}(I_T)}\int_{0}^{t}\!\!\!\int_{\R}(\px\ueb)^2dx - \frac{1}{3}\int_{\R}\ueb^3 dsdx\\
&\qquad \le C_{0} + \gamma C_0 C(T)T\norm{\ueb}_{L^{\infty}(I_T)} + C_0 \norm{\ueb}_{L^{\infty}(I_T)}+\frac{1}{3}\left\vert\int_{\R}\ueb^3dx\right\vert\\
&\qquad \le C_{0} + \gamma C_0 C(T)T\norm{\ueb}_{L^{\infty}(I_T)} +C_0\norm{\ueb}_{L^{\infty}(I_T)}+\frac{1}{3}\norm{\ueb}_{L^{\infty}(I_T)}\int_{\R}(\ueb)^2 dx \\
&\qquad \le C_{0} + \gamma C_0 C(T)T\norm{\ueb}_{L^{\infty}(I_T)} +C_0\norm{\ueb}_{L^{\infty}(I_T)}+\frac{1}{3}C_0 \norm{\ueb}_{L^{\infty}(I_T)}\\
&\qquad =  C_{0} + C_0 \left(\gamma C(T)T + \frac{4}{3}\right)\norm{\ueb}_{L^{\infty}(I_T)}.
\end{align*}
Hence,
\begin{equation}
\label{eq:ux1}
\begin{split}
\beta\norm{\px\ueb(t,\cdot)}^2_{L^2(\R)}&+ 2\beta\eps\int_{0}^{t}\norm{\pxx\ueb(s,\cdot)}^2_{L^2(\R)}ds\\
\le & C(T)\left(1+\norm{\ueb}_{L^{\infty}(I_T)}\right).
\end{split}
\end{equation}
Following \cite{CK, LN}, we begin by observing that, due to \eqref{eq:l2-u}, \eqref{eq:ux1} and the H\"older inequality,
\begin{align*}
\ueb^2(t,x)&=2\int_{-\infty}^{x}\ueb\px\ueb dy\le 2\int_{\R}\vert\ueb\px\ueb\vert dx\\
&\le\frac{2}{\sqrt{\beta}}\norm{\ueb}_{L^2(\R)}\sqrt{\beta}\norm{\px\ueb(t,\cdot)}_{L^2(\R)}\\
&\le\frac{2}{\sqrt{\beta}}C_0\sqrt{ C(T)\left(1+\norm{\ueb}_{L^{\infty}(I_T)}\right)},
\end{align*}
that is
\begin{equation}
\label{eq:quarto-grado}
\norm{\ueb}^4_{L^{\infty}(I_T)} \le \frac{C(T)}{\beta}\left( 1+\norm{\ueb}_{L^{\infty}(I_T)}\right).
\end{equation}
Let us show that \eqref{eq:quarto-grado} is verified when
\begin{equation}
\label{eq:sol-eq-quarto}
\norm{\ueb}_{L^{\infty}(I_T)}\le \max\left\{1, \left(\frac{2C(T)}{\beta}\right)^{\frac{1}{3}}\right\}\le C(T)\beta^{-\frac{1}{3}}.
\end{equation}
Let us  consider $y= \norm{\ueb}_{L^{\infty}(I_T)}$. Therefore, it follows from \eqref{eq:quarto-grado} that
\begin{equation}
\label{eq:equat-in-y}
y^4 \le  \frac{C(T)}{\beta}\left( 1+y \right).
\end{equation}
If $y\le 1$, we have \eqref{eq:sol-eq-quarto}.\\
Let $y>1$ and  we suppose that
\begin{equation*}
y>\left(\frac{2C(T)}{\beta}\right)^{\frac{1}{3}},
\end{equation*}
that is
\begin{equation}
\label{eq:potenza}
y^3>\frac{2C(T)}{\beta},  \quad y^4>\frac{2C(T)}{\beta}y.
\end{equation}
It follows from \eqref{eq:potenza} that
\begin{equation*}
2y^4>y^4+y^3>\frac{2C(T)}{\beta}(y+1).
\end{equation*}
Therefore,
\begin{equation*}
y^4>\frac{C(T)}{\beta}(y+1),
\end{equation*}
which is in contradiction with \eqref{eq:equat-in-y}.

Finally, \eqref{eq:ux2} follows from \eqref{eq:l-infty-u} and \eqref{eq:ux1}.
\end{proof}

\begin{lemma}
\label{lm:bounded}
Let  $T>0$. Assume \eqref{eq:beta-eps} holds true. Then:
\begin{itemize}
\item[$i$)] the family $\{\ueb\}_{\eps,\beta}$ is bounded in $L^{4}(I_T)$;
\item[$ii$)] the following families $\{\beta\pxx\ueb\}_{\eps,\beta},\,\{\sqrt{\eps}\ueb\px\ueb\}_{\eps,\beta}, \, \{\beta\sqrt{\eps}\pxxx\ueb \}_{\eps,\beta} $ are bounded in $L^2(I_T)$,
 \end{itemize}
where $I_T$ is defined in \eqref{eq:defI}.
\end{lemma}
\begin{proof}
Let $0\le t \le T$. Multiplying  \eqref{eq:OHepsw} by
\begin{equation*}
\ueb^3 - 3\beta(\px\ueb)^2 - 6\beta\ueb\pxx\ueb +\frac{18}{5}\beta^2\pxxxx\ueb ,
\end{equation*}
we have
\begin{equation}
\label{eq:Ohmp-1}
\begin{split}
\Big(\ueb^3 &- 3\beta(\px\ueb)^2 - 6\beta\ueb\pxx\ueb +\frac{18}{5}\beta^2\pxxxx\ueb\Big)\pt\ueb\\
&+ \Big(\ueb^3 - 3\beta(\px\ueb)^2 - 6\beta\ueb\pxx\ueb +\frac{18}{5}\beta^2\pxxxx\ueb\Big)\ueb\px\ueb\\
&-\beta\Big(\ueb^3 - 3\beta(\px\ueb)^2 - 6\beta\ueb\pxx\ueb +\frac{18}{5}\beta^2\pxxxx\ueb\Big)\pxxx\ueb\\
&-\eps\Big(\ueb^3 - 3\beta(\px\ueb)^2 - 6\beta\ueb\pxx\ueb +\frac{18}{5}\beta^2\pxxxx\ueb\Big)\pxx\ueb\\
=&\gamma\Big(\ueb^3 - 3\beta(\px\ueb)^2 - 6\beta\ueb\pxx\ueb +\frac{18}{5}\beta^2\pxxxx\ueb\Big)\Peb.
\end{split}
\end{equation}
We observe that
\begin{align*}
&-3\beta\int_{\R}(\px\ueb)^2\pt\ueb dx -6\beta\int_{\R} \ueb\pxx\ueb \pt\ueb dx\\
&\qquad =-3\beta\int_{\R}(\px\ueb)^2\pt\ueb dx +6\beta\int_{\R}\px\ueb\px(\ueb\pt\ueb) dx\\
&\qquad =3\beta\int_{\R}(\px\ueb)^2\pt\ueb dx + 6\beta \int_{\R}\ueb\px\ueb\ptx\ueb dx\\
&\qquad = 3\beta\int_{\R} (\px\ueb)^2\pt\ueb dx + 3\beta\int_{\R}\ueb\pt(\px\ueb)^2dx,
\end{align*}
that is
\begin{equation}
\label{eq:u-ux}
\begin{split}
-3\beta\int_{\R}(\px\ueb)^2\pt\ueb dx &-6\beta\int_{\R} \ueb\pxx\ueb \pt\ueb dx\\
=&3\beta\frac{d}{dt}\int_{\R}\ueb(\px\ueb)^2 dx.
\end{split}
\end{equation}
We have that
\begin{align*}
&-3\beta\int_{\R}(\px\ueb)^3\ueb dx\\
&\qquad  =3\beta\int_{\R}\ueb\px(\ueb(\px\ueb)^2)dx\\
&\qquad =3\beta\int_{\R} (\px\ueb)^3\ueb dx + 6\beta\int_{\R}\ueb^2 \px\ueb\pxx\ueb dx,
\end{align*}
namely
\begin{equation}
\label{eq:u-ux-3}
-\beta\int_{\R}(\px\ueb)^3\ueb dx= \beta\int_{\R}\ueb^2 \px\ueb\pxx\ueb dx.
\end{equation}
We observe that
\begin{equation}
\label{eq:ux-uxx}
3\beta^2\int_{\R}(\px\ueb)^2 \pxxx\ueb dx = - 6\beta^2\int_{\R}\px\ueb(\pxx\ueb)^2 dx.
\end{equation}
Moreover,
\begin{align*}
&6\beta^2\int_{\R}\ueb\pxx\ueb\pxxx\ueb dx\\
&\qquad = -6\beta^2\int_{\R}\pxx\ueb\px(\ueb\pxx\ueb)dx\\
&\qquad =- 6\beta^2 \int_{\R} (\pxx\ueb)^2 \px\ueb dx -6\beta^2 \int_{\R}\ueb\pxx\ueb\pxxx\ueb dx,
\end{align*}
that is
\begin{equation}
\label{eq:u-ux-uxx}
\beta^2\int_{\R}\ueb\pxx\ueb\pxxx\ueb dx=-\frac{1}{2}\beta^2 \int_{\R}\px\ueb (\pxx\ueb)^2dx.
\end{equation}
Due to \eqref{eq:ux-uxx} and \eqref{eq:u-ux-uxx},
\begin{align*}
&\frac{18}{5}\beta^2\int_{\R}\ueb\px\ueb\pxxxx\ueb dx \\
&\qquad= -\frac{18}{5}\beta^2\int_{\R}\pxxx\ueb\px(\ueb\px\ueb) dx\\
& \qquad = -\frac{18}{5}\beta^2\int_{\R}(\px\ueb)^2\pxxx\ueb dx-\frac{18}{5}\beta^2\int_{\R}\ueb\pxx\ueb\pxxx\ueb\\
& \qquad = \frac{36}{5}\beta^2\int_{\R}(\pxx\ueb)^2\px\ueb dx + \frac{9}{5}\beta^2\int_{\R}(\pxx\ueb)^2\px\ueb dx.
\end{align*}
Hence,
\begin{equation}
\label{eq:u-ux-uxxxx}
\frac{18}{5}\beta^2\int_{\R}\ueb\px\ueb\pxxxx\ueb dx= 9\beta^2\int_{\R}(\pxx\ueb)^2\px\ueb dx.
\end{equation}
It follows from \eqref{eq:OHepsw} that
\begin{equation}
\label{eq:u-xxxx-p}
\begin{split}
&\frac{18}{5}\beta^2\int_{\R}\Peb \pxxxx\ueb dx = -\frac{18}{5}\beta^2\int_{\R}\px\Peb\pxxx\ueb  dx\\
&\qquad =- \frac{18}{5}\beta^2\int_{\R}\ueb\pxxx\ueb dx= \frac{18}{5}\beta^2\int_{\R}\px\ueb\pxx\ueb dx=0.
\end{split}
\end{equation}
Since
\begin{align*}
\int_{\R}\ueb^3 \pt\ueb dx &=\frac{1}{4} \frac{d}{dt}\int_{\R} \ueb^4 dx,\\
\frac{18}{5}\beta^2\int_{\R}\pxxxx\ueb\pt\ueb dx &= \frac{9}{5}\beta^2\frac{d}{dt}\int_{\R}(\pxx\ueb)^2 dx,\\
-\beta\int_{\R}\ueb^3\pxxx\ueb dx&=3\beta\int_{\R}\ueb^2\px\ueb\pxx\ueb dx, \\
-3\eps\int_{\R}\ueb^3\pxx\ueb dx &=3\eps\int_{\R}\ueb^2(\px\ueb)^2 dx,\\
6\eps\beta\int_{\R}\ueb(\pxx\ueb)^2 dx &= -6\eps\beta \int_{\R}\ueb\px\ueb\pxxx\ueb dx,\\
-\frac{18}{5}\eps\beta^2\int_{\R}\pxxxx\ueb\pxx\ueb dx &= \frac{18}{5}\eps\beta^2\int_{\R}(\pxxx\ueb)^2 dx,
\end{align*}
for \eqref{eq:u-ux}, \eqref{eq:u-ux-3}, \eqref{eq:ux-uxx}, \eqref{eq:u-ux-uxx}, \eqref{eq:u-ux-uxxxx} and \eqref{eq:u-xxxx-p}, integrating \eqref{eq:Ohmp-1} over $\R$, we get
\begin{align*}
\frac{d}{dt}G(t)&+ 3\eps\int_{\R}\ueb^2(\px\ueb)^2 dx + \frac{18}{5}\eps\beta^2\int_{\R}(\pxxx\ueb)^2 dx\\
=& \gamma\int_{R}\ueb^3\Peb dx -3\gamma\beta\int_{\R}(\px\ueb)^2 \Peb dx\\
& - 6\gamma\beta\int_{\R}\ueb\pxx\ueb\Peb dx + 6\eps\beta\int_{\R}\ueb\px\ueb\pxxx\ueb dx\\
\le & \gamma\int_{\R}\vert \ueb^3\vert\vert\Peb\vert dx + 3\gamma\beta\int_{\R} (\px\ueb)^2 \vert\Peb\vert dx\\
&+6\gamma\beta\int_{\R}\vert\ueb\vert\vert\pxx\ueb\vert\vert\Peb\vert +6\eps\beta\int_{\R}\vert\ueb\px\ueb\vert\vert\pxxx\ueb\vert dx,
\end{align*}
where
\begin{equation}
\label{eq:def-di-G}
G(t)=\frac{1}{4}\int_{\R}\ueb^4 dx + 3\beta\int_{\R}\ueb(\px\ue)^2 dx + \frac{9}{5}\beta^2\int_{\R}(\pxx\ueb)^2 dx.
\end{equation}
The Young's inequality gives
\begin{align*}
&6\eps\beta\int_{\R}\vert\ueb\px\ueb\vert\vert\pxxx\ueb\vert dx\\
&\qquad=6\eps\int_{\R}\left\vert\beta\sqrt{D_1}\pxxx\ueb\right\vert\left\vert \frac{\ueb\px\ueb}{\sqrt{D_1}}\right\vert dx\\
&\qquad\le 3D_{1}\eps\beta^2\int_{\R}(\pxxx\ueb)^2 + \frac{3\eps}{D_1}\int_{\R}\ueb^2(\px\ueb)^2 dx,
\end{align*}
where $D_1$ is a positive constant.

Therefore,
\begin{align*}
\frac{d}{dt}G(t)&+ 3\eps\left(1 - \frac{1}{D_1}\right)\int_{\R}\ueb^2(\px\ueb)^2 dx\\
 &+\beta^2\eps\left(\frac{18}{5} - 3D_{1}\right)\int_{\R}(\pxxx\ueb)^2 dx\\
\le & \gamma\int_{\R}\vert \ueb^3\vert\vert\Peb\vert dx + 3\gamma\beta\int_{\R} (\px\ueb)^2 \vert\Peb\vert dx\\ &+6\gamma\beta\int_{\R}\vert\ueb\vert\vert\pxx\ueb\vert\vert\Peb\vert dx.
\end{align*}
Choosing $\displaystyle D_{1}\in\left(1, \frac{6}{5}\right)$, for \eqref{eq:P-infty}, we have that
\begin{align*}
\frac{d}{dt}G(t)&+ 3 D_{2}\eps\int_{\R}\ueb^2(\px\ueb)^2 dx +\eps \beta^2 D_{3}\int_{\R}(\pxxx\ueb)^2 dx\\
\le &\gamma\int_{\R}\vert \ueb^3\vert\vert\Peb\vert dx + 3\gamma C(T)\beta\int_{\R} (\px\ueb)^2 dx\\ &+6\gamma\beta\int_{\R}\vert\ueb\vert\vert\pxx\ueb\vert\vert\Peb\vert dx,
\end{align*}
where $D_2,\, D_{3}$ are two fixed positive constants.

Due to \eqref{eq:l2-u}, \eqref{eq:P-infty} and Young's inequality, we obtain
\begin{align*}
&\gamma\int_{\R}\vert \ueb^3\vert\vert\Peb\vert dx dx+6\gamma\beta\int_{\R}\vert\ueb\vert\vert\pxx\ueb\vert\vert\Peb\vert\\
&\quad =\int_{\R}\left\vert\frac{1}{\sqrt{2}}\ueb^2\right\vert\left\vert \sqrt{2}\gamma\ueb\Peb\right\vert dx\\
&\qquad+ \int_{\R}\left\vert\frac{3\sqrt{2}}{\sqrt{5}}\beta\pxx\ueb\right\vert \left\vert \frac{\sqrt{5}\gamma}{\sqrt{2}}\ueb\Peb\right\vert dx\\
&\quad \leq\frac{1}{4}\int_{\R}\ueb^4 dx + \gamma^2\int_{\R}\ueb^2\Peb^2 dx\\
&\qquad+\frac{9}{5}\beta^2\int_{\R}(\pxx\ueb)^2 dx+ \frac{5\gamma^2}{4}\int_{\R}\ueb^2\Peb^2dx\\
&\quad=  \frac{1}{4}\int_{\R}\ueb^4 dx + \frac{9}{5}\beta^2\int_{\R}(\pxx\ueb)^2 dx + \frac{9\gamma^2}{4}\int_{\R}\ueb^2\Peb^2dx\\
&\quad\leq  \frac{1}{4}\int_{\R}\ueb^4 dx + \frac{9}{5}\beta^2\int_{\R}(\pxx\ueb)^2 dx + C(T).
\end{align*}
Thus,
\begin{equation}
\label{eq:Equat-di-G}
\begin{split}
\frac{d}{dt}G(t)&+ 3 \eps D_{2}\int_{\R}\ueb^2(\px\ueb)^2 dx + \eps\beta D_{3}^2\int_{\R}(\pxxx\ueb)^2 dx\\
\le & \frac{1}{4}\int_{\R}\ueb^4 dx + \frac{9}{5}\beta^2\int_{\R}(\pxx\ueb)^2 dx + 3\gamma C(T)\beta\int_{\R} (\px\ueb)^2 dx + C(T).
\end{split}
\end{equation}
Due to \eqref{eq:def-di-G}, we have that
\begin{align*}
\frac{d}{dt}G(t) & + 3 D_{2}\eps\int_{\R}\ueb^2(\px\ueb)^2 dx +\eps\beta^2 D_{3}\int_{\R}(\pxxx\ueb)^2 dx\\
\leq &\frac{1}{4}\int_{\R}\ueb^4 dx + \frac{9}{5}\beta^2\int_{\R}(\pxx\ueb)^2 dx +3\beta\gamma C(T)\int_{\R} (\px\ueb)^2 dx + C(T)\\
= &\frac{1}{4}\int_{\R}\ueb^4 dx + \frac{9}{5}\beta^2\int_{\R}(\pxx\ueb)^2 dx +3\beta\int_{\R}\ueb(\px\ueb)^2 dx\\
&- 3\beta\int_{\R}\ueb(\px\ueb)^2 dx + 3\beta\gamma C(T)\int_{\R} (\px\ueb)^2 dx + C(T)\\
=& G(t) - 3\beta\int_{\R}\ueb(\px\ueb)^2 dx + 3\beta\gamma C(T)\int_{\R} (\px\ueb)^2 dx + C(T),
\end{align*}
that is
\begin{align*}
\frac{d}{dt}G(t)& -G(t)  + 3\eps D_{2}\int_{\R}\ueb^2(\px\ueb)^2 dx +\eps\beta^2 D_{3}\int_{\R}(\pxxx\ueb)^2 dx\\
\le & - 3\beta\int_{\R}\ueb(\px\ueb)^2 dx + 3\beta\gamma C(T)\int_{\R} (\px\ueb)^2 dx + C(T).
\end{align*}
Hence,
\begin{align*}
\frac{d}{dt}G(t)& -G(t)  + 3\eps D_{2}\int_{\R}\ueb^2(\px\ueb)^2 dx +\eps\beta^2 D_{3}\int_{\R}(\pxxx\ueb)^2 dx\\
\le & \left \vert- 3\beta\int_{\R}\ueb(\px\ueb)^2 dx + 3\beta\gamma C(T)\int_{\R} (\px\ueb)^2 dx + C(T) \right\vert\\
\le & 3\beta\left\vert \int_{\R}\ueb(\px\ueb)^2 dx\right\vert + 3\beta\gamma C(T)\int_{\R} (\px\ueb)^2 dx + C(T)\\
\le & 3\beta \int_{\R}\vert\ueb\vert(\px\ueb)^2 dx+ 3\beta\gamma C(T)\int_{\R} (\px\ueb)^2 dx + C(T).
\end{align*}
Gronwall's Lemma and \eqref{eq:u0eps} give
\begin{align*}
G(t)&+ 3\eps D_{2} e^{t}\int_{0}^{t}e^{-s}\norm{\ueb(s,\cdot)\px\ueb(s,\cdot)}^2_{L^2(\R)} ds\\
&+ \eps \beta^2 D_{3} e^{t}\int_{0}^{t}e^{-s}\norm{\pxxx\ueb(s,\cdot)}^2_{L^2(\R)}ds\\
\leq & C_{0}e^{t}+3\beta e^{t} \int_{0}^{t}e^{-s} \int_{\R} \vert \ueb \vert (\px\ue)^2 dsdx\\
&+ 3\gamma\beta C(T) e^{t}\int_{0}^{t} e^{-s} \norm{\px\ueb(s,\cdot)}^2_{L^2(\R)} ds + C(T)e^{t}\int_{0}^{t} e^{-s}ds\\
\leq &  C_0 e^{T} +3\beta e^{T} \int_{0}^{t}\!\!\!\int_{\R} \vert \ueb \vert (\px\ue)^2 dsdx\\
& + 3\beta\gamma C(T) e^{t}\int_{0}^{t} \norm{\px\ueb(s,\cdot)}^2_{L^2(\R)} ds + C(T)Te^{T}.
\end{align*}
Due to  \eqref{eq:beta-eps}, \eqref{eq:l2-u} and \eqref{eq:l-infty-u},  we have that
\begin{equation}
\begin{split}
\label{eq:l-2-u-ux}
3\beta\int_{0}^{t}\!\!\!\int_{\R} \vert \ueb \vert (\px\ueb)^2 dsdx\le&3\beta\norm{\ueb}_{L^{\infty}(I_T)}\int_{0}^{t}\norm{\px\ueb(s,\cdot)}^2_{L^2(\R))} ds\\
=&3\frac{\beta\eps}{\eps}\norm{\ueb}_{L^{\infty}(I_T)}\int_{0}^{t}\norm{\px\ueb(s,\cdot)}^2_{L^2(\R)} ds\\
\le& 3\frac{\beta^{\frac{2}{3}}}{\eps}C_0 C(T)\le C(T).
\end{split}
\end{equation}
Again by \eqref{eq:beta-eps}, \eqref{eq:l2-u} and \eqref{eq:l-infty-u},
\begin{equation}
\begin{split}
\label{eq:L-23}
3\beta\gamma C(T)\int_{0}^{t}\norm{\px\ueb(s,\cdot)}^2_{L^2(\R)}ds &= 3\frac{\beta\eps}{\eps}\gamma C(T) \int_{0}^{t}\norm{\px\ueb(s,\cdot)}^2_{L^2(\R)}ds\\
& \leq 3\frac{\beta}{\eps}\gamma C(T)C_0\le C(T).
\end{split}
\end{equation}
Hence, it follows from \eqref{eq:def-di-G}, \eqref{eq:l-2-u-ux} and \eqref{eq:L-23} that
\begin{align*}
\frac{1}{4}\int_{\R}\ueb^4 dx & + 3\beta\int_{\R}\ueb(\px\ue)^2 dx + \frac{9}{5}\beta^2\int_{\R}(\pxx\ueb)^2 dx\\
&+3\eps D_{2} e^{t}\int_{0}^{t}e^{-s}\norm{\ueb(s,\cdot)\px\ueb(s,\cdot)}^2_{L^2(\R)} ds \\
&+\eps\beta^2 D_{3} e^{t}\int_{0}^{t}e^{-s}\norm{\pxxx\ueb(s,\cdot)}^2_{L^2(\R)}ds\le C(T).
\end{align*}
Therefore,
\begin{align*}
\frac{1}{4}\int_{\R}\ueb^4 dx &  + \frac{9}{5}\beta^2\int_{\R}(\pxx\ueb)^2 dx\\
&+3\eps D_{2} e^{t}\int_{0}^{t}e^{-s}\norm{\ueb(s,\cdot)\px\ueb(s,\cdot)}^2_{L^2(\R)} ds \\
&+\eps\beta^2 D_{3} e^{t}\int_{0}^{t}e^{-s}\norm{\pxxx\ueb(s,\cdot)}^2_{L^2(\R)}ds\\
\le & C(T) - 3\beta\int_{\R}\ueb(\px\ue)^2 dx\le  \left\vert C(T) - 3\beta\int_{\R}\ueb(\px\ue)^2 dx\right\vert\\
\le & C(T) +3\beta \left\vert  \int_{\R}\ueb(\px\ue)^2 dx \right\vert \le  C(T) + 3\beta \int_{\R}\vert\ueb\vert(\px\ue)^2 dx,
\end{align*}
that is
\begin{align*}
\frac{1}{4}\int_{\R}\ueb^4 dx &  + \frac{9}{5}\beta^2\int_{\R}(\pxx\ueb)^2 dx\\
&+3\eps D_{2} e^{t}\int_{0}^{t}e^{-s}\norm{\ueb(s,\cdot)\px\ueb(s,\cdot)}^2_{L^2(\R)} ds \\
&+\eps\beta^2 D_{3} e^{t}\int_{0}^{t}e^{-s}\norm{\pxxx\ueb(s,\cdot)}^2_{L^2(\R)}ds\\
\le& C(T) + 3\beta \int_{\R}\vert\ueb\vert(\px\ue)^2 dx.
\end{align*}
An integration on $(0,T)$ and \eqref{eq:l-2-u-ux} give
\begin{align*}
\frac{1}{4}\int_{0}^{T}\!\!\!\!\int_{\R}\ueb^4 dtdx &+ \frac{9}{5}\beta^2\int_{0}^{T}\!\!\!\!\int_{\R}(\pxx\ueb)^2 dtdx\\
&+3\eps D_{2}\int_{0}^{T} e^{t}\int_{0}^{t}e^{-s}\norm{\ueb(s,\cdot)\px\ueb(s,\cdot)}^2_{L^2(\R)} dtds\\
&+\eps\beta^2 D_{3}\int_{0}^{T}e^{t}\int_{0}^{t}e^{-s}\norm{\pxxx\ueb(s,\cdot)}^2_{L^2(\R)}dtds\le  C(T).
\end{align*}
Therefore,
\begin{equation}
\label{eq:L-quarto}
\frac{1}{4}\int_{0}^{T}\!\!\!\!\int_{\R}\ueb^4 dtdx + \frac{9}{5}\beta^2\int_{0}^{T}\!\!\!\!\int_{\R}(\pxx\ueb)^2 dtdx\leq C(T),
\end{equation}
that is
\begin{align*}
\frac{1}{4}\int_{0}^{T}\!\!\!\!\int_{\R}\ueb^4 dtdx & \le C(T),\\
\frac{9}{5}\beta^2\int_{0}^{T}\!\!\!\!\int_{\R}(\pxx\ueb)^2 dtdx &\le C(T).
\end{align*}
It follows from \eqref{eq:def-di-G}, \eqref{eq:Equat-di-G}, \eqref{eq:L-23}, \eqref{eq:L-quarto} and an integration on $(0,t)$ that
\begin{align*}
\frac{1}{4}\int_{\R}\ueb^4 dx & + 3\beta\int_{\R}\ueb(\px\ue)^2 dx + \frac{9}{5}\beta^2\int_{\R}(\pxx\ueb)^2 dx\\
&+3\eps D_{2}\int_{0}^{t}\norm{\ueb(s,\cdot)\px\ueb(s,\cdot)}^2_{L^2(\R)} ds\\
& +\eps\beta^2 D_{3}\int_{0}^{t}\norm{\pxxx\ueb(s,\cdot)}^2_{L^2(\R)}ds\le C(T).
\end{align*}
Thus,
\begin{align*}
\frac{1}{4}\int_{\R}\ueb^4 dx & + 3\beta\int_{\R}\ueb(\px\ue)^2 dx + \frac{9}{5}\beta^2\int_{\R}(\pxx\ueb)^2 dx\\
&+3\eps D_{2}\int_{0}^{t}\norm{\ueb(s,\cdot)\px\ueb(s,\cdot)}^2_{L^2(\R)} ds\\
& +\eps\beta^2 D_{3}\int_{0}^{t}\norm{\pxxx\ueb(s,\cdot)}^2_{L^2(\R)}ds\\
\le & C(T) - 3\beta\int_{\R}\ueb(\px\ue)^2dx\le  \left\vert C(T) - 3\beta\int_{\R}\ueb(\px\ue)^2dx\right\vert\\
\le & C(T) + 3\beta\left\vert \int_{\R}\ueb(\px\ue)^2dx\right\vert\le  C(T) + 3\beta \int_{\R}\vert\ueb\vert(\px\ue)^2dx.
\end{align*}
A new integration on $(0,T)$ and \eqref{eq:l-2-u-ux} give
\begin{align*}
\frac{1}{4}\int_{0}^{T}\!\!\!\int_{\R}\ueb^4 dtdx &+ \frac{9}{5}\beta^2\int_{0}^{T}\!\!\!\int_{\R}(\pxx\ueb)^2 dtdx\\
&+3\eps D_{2}\int_{0}^{T}\!\!\!\int_{0}^{t}\norm{\ueb(s,\cdot)\px\ueb(s,\cdot)}^2_{L^2(\R)} dtds\\ &+\eps\beta^2D_{3}\int_{0}^{T}\!\!\!\int_{0}^{t}\norm{\pxxx\ueb(s,\cdot)}^2_{L^2(\R)} dtds\le C(T).
\end{align*}
Hence,
\begin{equation*}
\begin{split}
\eps\int_{0}^{T}\norm{\ueb(t,\cdot)\px\ueb(t,\cdot)}^2_{L^2(\R)} dt&\le C(T),\\
\beta^2\eps\int_{0}^{T}\norm{\pxxx\ueb(t,\cdot)}^2_{L^2(\R)} dt &\le C(T).
\end{split}
\end{equation*}
The proof is done.
\end{proof}

\begin{lemma}
\label{lm:501}
Let  $T>0$. Assume that \eqref{eq:beta-eps} holds true. Then:
\begin{itemize}
\item[$i$)] the family $\{\eps\px\ueb\}_{\eps,\beta}$ is bounded in $L^{\infty}((0,T);L^2(\R))$;
\item[$ii$)] the family $\{\eps\sqrt{\eps}\pxx\ueb\}_{\eps,\beta}$ is bounded in $L^2(I_T)$;
\item[$iii$)] the family $\{\beta\px\ueb\pxx\ueb\}_{\eps,\beta}$ is bounded in $L^1(I_T)$,
\end{itemize}
where $I_T$ is defined in \eqref{eq:defI}.\\
Moreover,
\begin{equation}
\label{eq:defuxx}
\beta^2\int_{0}^{T}\norm{\pxx\ueb(s,\cdot)}^2_{L^2(\R)}ds \le C(T)\eps.
\end{equation}
\end{lemma}
\begin{proof}
Let $0\le t\le T$. Multiplying \eqref{eq:OHepsw} by $-\eps^2\pxx\ueb$, we have that
\begin{equation}
\label{eq:Ohmp-122}
\begin{split}
-\eps^2\pxx\ueb\pt\ueb& -\eps^2 \ueb\px\ueb\pxx\ueb\\
 &+ \eps^2\beta\pxxx\ueb\pxx\ueb +\eps^3(\pxx\ueb)^2=-\eps^2\gamma\Peb\pxx\ueb.
\end{split}
\end{equation}
Since
\begin{align*}
-\eps^2\int_{\R}\pxx\ueb\pt\ueb dx &=\frac{\eps^2}{2}\frac{d}{dt}\int_{\R}(\px\ueb)^2dx,\\
\eps^2\beta\int_{\R}\pxx\ueb\pxxx\ueb dx& =0,
\end{align*}
integrating \eqref{eq:Ohmp-122} over $\R$,
\begin{align*}
\frac{\eps^2}{2}\frac{d}{dt}\int_{\R}(\px\ueb)^2dx &+\eps^3\int_{\R}(\pxx\ueb)^2dx\\
=& \eps^2 \int_{\R}\ueb\px\ueb\pxx\ueb dx -\eps^2\gamma \int_{\R}\Peb\pxx\ueb dx.
\end{align*}
For \eqref{eq:OHepsw},
\begin{align*}
-\eps^2\gamma \int_{\R}\Peb\pxx\ueb dx&=\eps^2\gamma\int_{\R}\px\Peb\px\ueb dx\\
& =\eps^2\gamma\int_{\R}\ueb\px\ueb dx=0.
\end{align*}
Therefore,
\begin{equation*}
\eps^2\frac{d}{dt}\int_{\R}(\px\ueb)^2dx +2\eps^3\int_{\R}(\pxx\ueb)^2dx=2\eps^2 \int_{\R}\ueb\px\ueb\pxx\ueb dx.
\end{equation*}
Due to Young's inequality,
\begin{align*}
&2\eps^2 \int_{\R}\ueb\px\ueb\pxx\ueb dx\\
&\quad\le 2\eps^2 \left\vert \int_{\R}\ueb\px\ueb\pxx\ueb dx \right\vert\\
& \quad \le 2\int_{\R} \eps^{\frac{1}{2}}\vert \ueb\px\ueb \vert \eps^{\frac{3}{2}}\vert\pxx\ueb\vert dx\\
&\quad \le \eps\int_{\R} (\ueb\px\ueb)^2 dx +\eps^3\int_{\R} (\pxx\ueb)^2 dx.
\end{align*}
Thus,
\begin{equation*}
\eps^2\frac{d}{dt}\int_{\R}(\px\ueb)^2dx + \eps^3\int_{\R}(\pxx\ueb)^2dx \le \eps\int_{\R} (\ueb\px\ueb)^2 dx .
\end{equation*}
An integration on $(0,t)$, \eqref{eq:u0eps} and Lemma \ref{lm:bounded} give
\begin{equation*}
\eps^2\norm{\px\ueb(t, \cdot)}^2_{L^2(\R)} +\eps^3\int_{0}^{t}\norm{\pxx\ueb(s,\cdot)}^2_{L^2(\R)}ds \le C(T).
\end{equation*}
Hence,
\begin{equation}
\label{eq:10032}
\begin{split}
\eps^2\norm{\px\ueb(t, \cdot)}^2_{L^2(\R)}&\le C(T),\\
\eps^3\int_{0}^{t}\norm{\pxx\ueb(s,\cdot)}^2_{L^2(\R)}ds &\le C(T).
\end{split}
\end{equation}
Thanks to  \eqref{eq:beta-eps}, \eqref{eq:l2-u}, \eqref{eq:10032} and H\"older inequality,
\begin{align*}
&\beta\int_{0}^{T}\!\!\!\int_{\R}\vert\px\ueb\pxx\ueb\vert dsdx =\frac{\beta}{\eps^2}\int_{0}^{T}\!\!\!\int_{\R}\eps^{\frac{1}{2}}\vert\px\ueb\vert\eps^{\frac{3}{2}}\vert\pxx\ueb\vert dx\\
&\quad \le \frac{\beta}{\eps^2} \left(\eps \int_{0}^{T}\!\!\!\int_{\R}(\px\ueb)^2 dsdx\right)^{\frac{1}{2}}\left(\eps^3 \int_{0}^{T}\!\!\!\int_{\R}(\pxx\ueb)^2 dsdx\right)^{\frac{1}{2}}\\
&\quad \le C_0 C(T)\frac{\beta}{\eps^2}\le C(T).
\end{align*}
Due to \eqref{eq:beta-eps} and \eqref{eq:10032}, we have
\begin{align*}
\beta^2\int_{0}^{T}\norm{\pxx\ueb(s,\cdot)}^2_{L^2(\R)}ds \le C_{0}^2\eps^4\int_{0}^{T}\norm{\pxx\ueb(s,\cdot)}^2_{L^2(\R)}ds\le C(T)\eps,
\end{align*}
which gives \eqref{eq:defuxx}.
\end{proof}

\section{Proof of Theorem \ref{th:main}}
\label{sec:theor}
In this section, we prove Theorem \ref{th:main}. The following technical lemma is needed  \cite{Murat:Hneg}.
\begin{lemma}
\label{lm:1}
Let $\Omega$ be a bounded open subset of $
\R^2$. Suppose that the sequence $\{\mathcal
L_{n}\}_{n\in\mathbb{N}}$ of distributions is bounded in
$W^{-1,\infty}(\Omega)$. Suppose also that
\begin{equation*}
\mathcal L_{n}=\mathcal L_{1,n}+\mathcal L_{2,n},
\end{equation*}
where $\{\mathcal L_{1,n}\}_{n\in\mathbb{N}}$ lies in a
compact subset of $H^{-1}_{loc}(\Omega)$ and
$\{\mathcal L_{2,n}\}_{n\in\mathbb{N}}$ lies in a
bounded subset of $\mathcal{M}_{loc}(\Omega)$. Then $\{\mathcal
L_{n}\}_{n\in\mathbb{N}}$ lies in a compact subset of $H^{-1}_{loc}(\Omega)$.
\end{lemma}
Moreover, we consider the following definition.
\begin{definition}
A pair of functions $(\eta, q)$ is called an  entropy--entropy flux pair if $\eta :\R\to\R$ is a $C^2$ function and $q :\R\to\R$ is defined by
\begin{equation*}
q(u)=\int_{0}^{u} \eta'(\xi)f'(\xi)d\xi.
\end{equation*}
An entropy-entropy flux pair $(\eta,\, q)$ is called  convex/compactly supported if, in addition, $\eta$ is convex/compactly supported.
\end{definition}
We begin by proving the following result.
\begin{lemma}\label{lm:dist-solution}
Assume that \eqref{eq:assinit}, \eqref{eq:def-di-P0}, \eqref{eq:L-2P0}, \eqref{eq:u0eps}, and \eqref{eq:beta-eps} hold. Then for any compactly supported entropy--entropy flux pair $(\eta,\, q)$, there exist two sequences $\{\eps_{n}\}_{n\in\N}$, $\{\beta_{n}\}_{n\in\N}$, with $\eps_n, \beta_n \to 0$, and two limit functions
\begin{align*}
&u\in L^{\infty}(0,T; L^4(\R)\cap L^2(\R)),\\
&P\in L^{\infty}((0,T)\times\R)\cap L^{2}((0,T)\times\R),
\end{align*}
 such that
\begin{align}
\label{eq:con1}
& u_{\eps_n, \beta_n}\to u \quad  \textrm{in} \quad  L^{p}_{loc}((0,T)\times\R),\quad \textrm{for each} \quad 1\le p <4,\\
\label{eq:con2}
&P_{\eps_n, \beta_n}\to P \quad \textrm{in} \quad L^{\infty}((0,T)\times\R)\cap L^{2}((0,T)\times\R),
\end{align}
and $u$ is a distributional solution of \eqref{eq:OHw}.
\end{lemma}

\begin{proof}
Let us consider a compactly supported entropy--entropy flux pair $(\eta, q)$. Multiplying \eqref{eq:OHepsw} by $\eta'(\ueb)$, we have
\begin{align*}
\pt\eta(\ueb) + \px q(\ueb) =&\eps \eta'(\ueb) \pxx\ueb + \beta \eta'(\ueb) \pxxx\ueb + \gamma \eta'(\ueb) \Peb\\
=& I_{1,\,\eps,\,\beta}+I_{2,\,\eps,\,\beta}+ I_{3,\,\eps,\,\beta} + I_{4,\,\eps,\,\beta}+I_{5,\,\eps,\,\beta},
\end{align*}
where
\begin{equation}
\begin{split}
\label{eq:12000}
I_{1,\,\eps,\,\beta}&=\px(\eps\eta'(\ueb)\px\ueb),\\
I_{2,\,\eps,\,\beta}&= -\eps\eta''(\ueb)(\px\ueb)^2,\\
I_{3,\,\eps,\,\beta}&= \px(\beta\eta'(\ueb)\pxx\ueb),\\
I_{4,\,\eps,\,\beta}&= -\beta\eta''(\ueb)\px\ueb\pxx\ueb,\\
I_{5,\,\eps,\,\beta}&=\gamma\eta'(\ueb) \Peb.
\end{split}
\end{equation}
Let us  show that
\begin{equation*}
\label{eq:H1}
I_{1,\,\eps,\,\beta}\to0 \quad \text{in $H^{-1}((0,T) \times\R),\,T>0.$}
\end{equation*}
Thanks to Lemma \ref{lm:l2-u},
\begin{align*}
\norm{\eps\eta'(\ueb)\px\ueb}^2_{L^2((0,T)\times\R))}&\leq \norm{\eta'}^2_{L^{\infty}(\R)}\eps ^2\int_{0}^{T} \norm{\px\ueb(s,\cdot)}^2_{L^2(\R)}ds\\
&\leq \norm{\eta'}^2_{L^{\infty}(\R)}\frac{\eps C_{0}}{2} \to 0.
\end{align*}
We claim that
\begin{equation*}
\{I_{2,\,\eps,\,\beta}\}_{\eps,\,\beta >0} \quad\text{is bounded in $L^1((0,T)\times\R),\, T>0$}.
\end{equation*}
Again for Lemma \ref{lm:l2-u},
\begin{align*}
\norm{ \eps\eta''(\ueb)(\px\ueb)^2}_{L^1((0,T)\times\R)}& \leq \norm{\eta''}_{L^\infty(\R)}\eps\int_{0}^{T}\norm{\px\ueb(s,\cdot)}^2_{L^2(\R)}ds\\
&\leq\norm{\eta''}_{L^\infty (\R)}\frac{C_0}{2}.
\end{align*}
We have that
\begin{equation*}
I_{3,\,\eps,\,\beta}\to0 \quad \text{in $H^{-1}((0,T) \times\R),\,T>0.$}
\end{equation*}
Thanks to Lemma \ref{lm:501},
\begin{align*}
\norm{\beta^2\eta'(\ueb)\pxx\ueb}^2_{L^2((0,T)\times\R))}&\leq \norm{\eta'}^2_{L^{\infty}(\R)}\beta ^2\int_{0}^{T} \norm{\pxx\ueb(s,\cdot)}^2_{L^2(\R)}ds\\
&\leq \norm{\eta'}^2_{L^{\infty}(\R)}C(T)\eps \to 0.
\end{align*}
Let us  show that
\begin{equation*}
\{I_{4,\,\eps,\,\beta}\}_{\eps,\,\beta >0}\quad\text{is bounded in $L^1((0,T)\times\R),\, T>0$}.
\end{equation*}
Again for Lemma \ref{lm:501},
\begin{align*}
&\norm{\beta\eta''(\ueb)\px\ueb\pxx\ueb}_{L^1((0,T)\times\R)}\\
&\qquad\leq \norm{\eta''}_{L^{\infty}(\R)}\beta\int_{0}^{T} \norm{\px\ueb(s,\cdot)\pxx\ueb(s,\cdot)}_{L^1(\R)}ds\\
&\qquad\leq \norm{\eta''}_{L^{\infty}(\R)}C(T).
\end{align*}
We claim that
\begin{equation*}
\{I_{5,\,\eps,\,\beta}\}_{\eps,\,\beta >0}\quad\text{is bounded in $L^1_{loc}((0,T)\times\R),\, T>0$}.
\end{equation*}
Let $K$ be a compact subset of $(0,T)\times\R$.  Lemma \ref{lm:P-infty} gives
\begin{align*}
\norm{\gamma\eta'(\ueb)\Peb}_{L^1(K)}&=\gamma\int_{K}\vert\eta'(\ueb)\vert\vert\Peb\vert
dtdx\\
&\leq \gamma\norm{\eta'}_{L^{\infty}(\R)}\norm{\Peb}_{L^{\infty}((0,T)\times\R)}\vert K \vert .
\end{align*}
Therefore, Lemma \ref{lm:1} and the $L^p$ compensated compactness of \cite{SC} give \eqref{eq:con1}.
\eqref{eq:con2} follows from Lemma \ref{lm:P-infty}.

We conclude by proving that $u$ is a distributional solution of \eqref{eq:OHw}.
Let $ \phi\in C^{\infty}(\R^2)$ be a test function with
compact support. We have to prove that
\begin{equation}
\label{eq:k1}
\int_{0}^{\infty}\!\!\!\!\!\int_{\R}\left(u\pt\phi+\frac{u^2}{2}\px\phi\right)dtdx - \gamma\int_{0}^{\infty}\!\!\!\!\!\int_{\R} P\phi dtdx +\int_{\R}u_{0}(x)\phi(0,x)dx=0.
\end{equation}
We have that
\begin{align*}
\int_{0}^{\infty}\!\!\!&\!\!\int_{\R}\left(u_{\eps_{n}, \beta_{n}}\pt\phi+\frac{u^2_{\eps_n, \beta_{n}}}{2}\px\phi\right)dtdx  - \gamma\int_{0}^{\infty}\!\!\!\!\!\int_{\R} P_{\eps_{n},\beta_{n}}\phi dtdx +\int_{\R}u_{0,\eps_n,\beta_n}(x)\phi(0,x)dx\\
&+\eps_{n}\int_{0}^{\infty}\!\!\!\!\!\int_{\R}u_{\eps_{n},\beta_{n}}\pxx\phi dtdx + \eps_n\int_{0}^{\infty}u_{0,\eps_{n},\beta_{n}}(x)\pxx\phi(0,x)dx\\
&- \beta_n\int_{0}^{\infty}\!\!\!\!\int_{\R}u_{\eps_n,\beta_n}\pxxx\phi dt dx - \beta_n\int_{0}^{\infty}u_{0,\eps_n,\beta_n}(x)\pxxx\phi(0,x)dx=0.
\end{align*}
Therefore, \eqref{eq:k1} follows from \eqref{eq:con1}, and \eqref{eq:con2}.
\end{proof}
\begin{lemma}
\label{lm:entropy-solution}
Assume that \eqref{eq:assinit}, \eqref{eq:def-di-P0}, \eqref{eq:L-2P0}, \eqref{eq:u0eps}, and \eqref{eq:beta-eps-1} hold. Then,
\begin{align}
\label{eq:con3}
& u_{\eps_n, \beta_n}\to u \quad  \textrm{in} \quad  L^{p}_{loc}((0,T)\times\R),\quad \textrm{for each} \quad 1\le p <4,\\
\label{eq:con4}
&P_{\eps_n, \beta_n}\to P \quad \textrm{in} \quad L^{p}_{loc}(0,T;W^{1,\infty}(\R)\cap H^{1}(\R)), \quad \textrm{for each} \quad 1\le p <4,
\end{align}
where $u$ is  the unique entropy solution of \eqref{eq:OHw}.
\end{lemma}
\begin{proof}
Let us consider a compactly supported entropy--entropy flux pair $(\eta, q)$. Multiplying \eqref{eq:OHepsw} by $\eta'(\ueb)$, we obtain that
\begin{align*}
\pt\eta(\ueb) + \px q(\ueb) =&\eps \eta'(\ueb) \pxx\ueb + \beta \eta'(\ueb) \pxxx\ueb + \gamma \eta'(\ueb) \Peb\\
=& I_{1,\,\eps,\,\beta}+I_{2,\,\eps,\,\beta}+ I_{3,\,\eps,\,\beta} + I_{4,\,\eps,\,\beta}+I_{5,\,\eps,\,\beta},
\end{align*}
where $I_{1,\,\eps,\,\beta},\,I_{2,\,\eps,\,\beta},\, I_{3,\,\eps,\,\beta} ,\, I_{4,\,\eps,\,\beta}$ and $I_{5,\,\eps,\,\beta}$ are defined in \eqref{eq:12000}.

As in Lemma \ref{lm:dist-solution}, we obtain that $\{I_{1,\,\eps,\,\beta}\}_{\eps,\beta >0}\to0$ in $H^{-1}((0,T) \times\R)$, $\{I_{2,\,\eps,\,\beta}\}_{\eps,\beta >0}$ is bounded in $L^1((0,T)\times\R)$, $I_{3,\,\eps,\,\beta}\to0$ in $H^{-1}((0,T) \times\R)$ and $\{I_{5,\,\eps,\,\beta}\}_{\nu>0}$ is bounded in $L^1_{loc}((0,T)\times\R)$.\\
Let us show that
\begin{equation*}
I_{4,\,\eps,\,\beta}\to 0\quad\text{in $L^1((0,T)\times\R),\, T>0$}.
\end{equation*}
Thanks to \eqref{eq:beta-eps-1} and Lemma \ref{lm:501},
\begin{align*}
&\norm{\beta\eta''(\ueb)\px\ueb\pxx\ueb}_{L^1((0,T)\times\R)}\\
&\qquad\leq \norm{\eta''}_{L^{\infty}(\R)}\beta\int_{0}^{T} \norm{\px\ueb(s,\cdot)\pxx\ueb(s,\cdot)}_{L^1(\R)}ds\\
&\qquad\leq \norm{\eta''}_{L^{\infty}(\R)}C(T)\frac{\beta}{\eps^2}\to 0.
\end{align*}
Therefore, Lemma \ref{lm:1} gives \eqref{eq:con3}. 
\eqref{eq:con4} follows from \eqref{eq:con3}, the H\"older inequality, and the identities
\begin{equation*}
P_{\eps_n,\,\beta_n}(t,x)=\int_0^x u_{\eps_{n},\,\beta_{n}}(t,y)dy,\qquad \px P_{\eps_n,\,\beta_n}=u_{\eps_{n},\,\beta_{n}}.
\end{equation*}

We conclude by proving that $u$ is the unique entropy solution of \eqref{eq:OHw}.
Let us consider a compactly supported entropy--entropy flux pair $(\eta, q)$, and $\phi\in C^{\infty}_{c}((0,\infty)\times\R)$ a non--negative function. We have to prove that
\begin{equation}
\label{eq:u-entropy-solution}
\int_{0}^{\infty}\!\!\!\!\!\int_{\R}(\pt\eta(u)+ \px q(u))\phi dtdx - \gamma\int_{0}^{\infty}\!\!\!\!\!\int_{\R}\Peb\eta'(u)\phi dtdx \le0.
\end{equation}
We have that
\begin{align*}
&\int_{0}^{\infty}\!\!\!\!\!\int_{\R}(\px\eta(u_{\eps_{n},\,\beta_{n}})+\px q(u_{\eps_{n},\,\beta_{n}}))\phi dtdx - \gamma\int_{0}^{\infty}\!\!\!\!\!\int_{\R}P_{\eps_{n},\,\beta_{n}}\eta'(u_{\eps_{n},\,\beta_{n}})\phi dtdx\\
&\qquad=\eps_{n}\int_{0}^{\infty}\!\!\!\!\!\int_{\R}\px(\eta'(u_{\eps_{n},\,\beta_{n}})\px u_{\eps_{n},\,\beta_{n}})\phi dtdx -\eps_{n}\int_{0}^{\infty}\!\!\!\!\!\int_{\R} \eta''(u_{\eps_{n},\,\beta_{n}})(\px u_{\eps_{n},\,\beta_{n}})^2\phi dtdx\\
&\qquad\quad +\beta_{n}\int_{0}^{\infty}\!\!\!\!\!\int_{\R}\px(\eta'(u_{\eps_{n},\,\beta_{n}})\pxx u_{\eps_{n},\,\beta_{n}})\phi dtdx- +\beta_{n}\int_{0}^{\infty}\!\!\!\!\!\int_{\R}\eta''(u_{\eps_{n},\,\beta_{n}})\px u_{\eps_{n},\,\beta_{n}}\pxx u_{\eps_{n},\,\beta_{n}}\phi dtdx\\
&\qquad \le - \eps_{n}\int_{0}^{\infty}\!\!\!\!\!\int_{\R}\eta'(u_{\eps_{n},\,\beta_{n}})\px u_{\eps_{n},\,\beta_{n}}\px\phi dtdx - \beta_{n}\int_{0}^{\infty}\!\!\!\!\!\int_{\R}\eta'(u_{\eps_{n},\,\beta_{n}})\pxx u_{\eps_{n},\,\beta_{n}}\px\phi dtdx\\
&\qquad\quad - \beta_{n}\int_{0}^{\infty}\!\!\!\!\!\int_{\R}\eta''(u_{\eps_{n},\,\beta_{n}})\px u_{\eps_{n},\,\beta_{n}}\pxx u_{\eps_{n},\,\beta_{n}}\phi dtdx\\
&\qquad \le  \eps_{n}\int_{0}^{\infty}\!\!\!\!\!\int_{\R}\vert\eta'(u_{\eps_{n},\,\beta_{n}})\vert\vert\px u_{\eps_{n},\,\beta_{n}}\vert\vert\px\phi\vert dtdx +\beta_{n}\int_{0}^{\infty}\!\!\!\!\!\int_{\R}\vert\eta'(u_{\eps_{n},\,\beta_{n}})\vert\vert\pxx u_{\eps_{n},\,\beta_{n}}\vert\vert\px\phi\vert dtdx\\
&\qquad\quad +\beta_{n}\int_{0}^{\infty}\!\!\!\!\!\int_{\R}\vert\eta''(u_{\eps_{n},\,\beta_{n}})\vert\vert\px u_{\eps_{n},\,\beta_{n}}\pxx u_{\eps_{n},\,\beta_{n}}\vert\vert\phi\vert dtdx\\
&\qquad\le  \eps_{n} \norm{\eta'}_{L^{\infty}(\R)}\norm{\px u_{\eps_{n},\,\beta_{n}}}_{L^2(supp(\px\phi))}\norm{\px\phi}_{L^2(supp(\px\phi))}\\
&\qquad\quad+ \beta_{n} \norm{\eta'}_{L^{\infty}(\R)}\norm{\pxx u_{\eps_{n},\,\beta_{n}}}_{L^2(supp(\px\phi))}\norm{\px\phi}_{L^2(supp(\px\phi))}\\
&\qquad\quad +\beta_{n} \norm{\eta''}_{L^{\infty}(\R)}\norm{\phi}_{L^{\infty}(\R)}\norm{\px u_{\eps_{n},\,\beta_{n}}\pxx u_{\eps_{n},\,\beta_{n}}}_{L^1(supp(\px\phi))}\\
&\qquad\le  \eps_{n} \norm{\eta'}_{L^{\infty}(\R)}\norm{\px u_{\eps_{n},\,\beta_{n}}}_{L^2((0,T)\times\R)}\norm{\px\phi}_{L^2((0,T)\times\R)}\\
&\qquad\quad+ \beta_{n} \norm{\eta'}_{L^{\infty}(\R)}\norm{\pxx u_{\eps_{n},\,\beta_{n}}}_{L^2((0,T)\times\R)}\norm{\px\phi}_{L^2((0,T)\times\R)}\\
&\qquad\quad+\beta_{n} \norm{\eta''}_{L^{\infty}(\R)}\norm{\phi}_{L^{\infty}(\R^{+}\times\R)}\norm{\px u_{\eps_{n},\,\beta_{n}}\pxx u_{\eps_{n},\,\beta_{n}}}_{L^1((0,T)\times\R)}.
\end{align*}
\eqref{eq:u-entropy-solution} follows from \eqref{eq:beta-eps-1}, \eqref{eq:con3}, \eqref{eq:con4}, Lemmas \ref{lm:l2-u} and \ref{lm:501}.
\end{proof}
\begin{proof}[Proof of Theorem \ref{th:main}]
Theorem \ref{th:main} follows from Lemmas \ref{lm:dist-solution}, and \ref{lm:entropy-solution}.
\end{proof}

\end{document}